\documentclass{amsart}%
\usepackage{amssymb,amsmath,mathtools,xcolor,graphicx,xspace,colortbl,ragged2e,rotating}
\usepackage{amsfonts}
\usepackage{amsmath}
\usepackage{tabulary}
\usepackage{wrapfig}
\usepackage{extarrows}
\usepackage{amssymb}
\usepackage{graphicx}%
\providecommand{\U}[1]{\protect\rule{.1in}{.1in}}
\graphicspath{{untitled3_graphics/}{untitled3_tcache/}{untitled3_gcache/}}
\DeclareGraphicsExtensions{.pdf,.eps,.ps,.png,.jpg,.jpeg}
\theoremstyle{plain}
\newtheorem{theorem}{Theorem}
\newtheorem{lemma}{Lemma}

\newtheorem{proposition}{Proposition}

\theoremstyle{definition}
\newtheorem{definition}{Definition}

\newtheorem{example}{Example}

\theoremstyle{remark}

\numberwithin{equation}{section}
\begin{document}
\title[Short Title]{A quick distributional way to reproduce some results of the Riemann zeta function}
\author{Junfa Deng}
\address{Junfa Deng, School of Mathematics\\
Hefei University of Technology\\
Hefei 230009\\
China}

\author{Yunyun Yang}
\address{Yunyun Yang, School of Mathematics\\
Hefei University of Technology\\
Hefei 230009\\
China}

\email{yangyunyun@hfut.edu.cn}
\email{dengjunfa@mail.hfut.edu.cn}
\author{Hao Zhang}
\address{Hao Zhang, }
\email{821906929@qq.com}
\date{\today}
\keywords{Riemann zeta function, distribution, Ces\`aro limit, Dirac delta function, Hadamard finite part }

\begin{abstract}
The evaluation of the Riemann zeta function at negative integers is a classical result typically obtained through analytic continuation or contour integration. In this paper, we present a novel and concise derivation of these special values by employing the theory of Ces\`aro limit of distributions, a generalized limit concept developed by Estrada, Kanwal, and Fulling. We use this tool to give a quick proof of
the result that
\[
\zeta(-n)=-\frac{B_{n+1}}{n+1},
\]
for $n\in\mathbb{N}^+.$ We also give a short discussion on $\zeta^{\prime
}(\alpha)$ and compute the value of $\zeta^{\prime}(0)$.

\end{abstract}
\maketitle







\section{Introduction}

The Riemann zeta function plays a crucial role in number theory and
mathematical analysis, with significant applications in statistics and
physics. It is derived from the analytic continuation of the series
$\sum_{n=1}^{\infty}{\frac{1}{n^{s}}}$. However, the evaluation of the Riemann
zeta function is not straightforward, especially when dealing with divergent
series. For instance, $\zeta(-1)$ yields the value $-1/12$, whereas the
corresponding series before analytic continuation represents the sum of all
natural numbers, leading to the seemingly paradoxical equation
\begin{equation}
\label{1}1+2+3+\cdots+n+\cdots=-\frac{1}{12}.
\end{equation}
It is crucial to note that equation (\ref{1}) is meaningful only under the
framework of analytic continuation of the Riemann zeta function, as the
left-hand side is divergent in the conventional sense.

Summation of divergent series has always been a difficult problem in analysis.
Here, the \textquotedblleft summation\textquotedblright\ does not refer to the
sum of the original series in the usual sense, but a way to characterize the
\textquotedblleft quantity\textquotedblright\ of a divergent series, making
the divergent series summable under a certain definition. Such summation
methods are also called the generalized sums of divergent series. Besides the
Ces\`{a}ro mean summation, there are many traditional summation methods for
divergent series, such as the Cauchy summation method, the Abel summation
method, and the Ramanujan summation method, etc \cite{CB}. The generalized sum
of a divergent series is essentially an extension. That is to say, the
generalized sum of a divergent series must be meaningful in a certain sense.
After defining the summation method, it should also be applicable to
convergent series, ensuring the invariance of the sum of convergent series.
Take the series $\sum_{n=0}^{\infty}(-1)^{n}$ as an example, which is a
classic divergent series. Its partial sums have no limit, only upper and lower
limits. But we can try to consider the power series
\begin{equation}
1+q+q^{2}+\cdots+q^{n}+\cdots=\frac{1}{1-q}. \label{2}%
\end{equation}
Obviously, the power series is only meaningful when $|q|<1$. However, the
right-hand side of equation (\ref{2}) itself is a smooth function on
$\mathbb{R}\setminus\{1\}$, regardless of convergence. Thus we can directly
take $\frac{1}{1-q}$ as the generalized sum of the divergent power series when
$q\neq1$. Let $q=-1$, then we obtain
\[
\sum_{n=0}^{\infty}(-1)^{n}=\frac{1}{2}\quad(\infty).
\]
The most famous application of generalized sums in analytic number theory is
the evaluation problem of the Riemann zeta function. Among the summation
methods mentioned above, only Ramanujan summation has been successfully applied to evaluate the values of the Riemann zeta function at negative integers. However, Ramanujan summation is profound, involving
Bernoulli number corrections, extended integrals, and more. A systematic
theory can be found in reference \cite{CB}. 

Indeed, by employing the tools of distributions, one can concisely compute
certain values of the zeta function. In this paper, we use the tool of the
Ces\`{a}ro mean of distributions to rederive the values of the zeta function
at negative integers. This is a new proof and our main result.

Moreover, we use the same method to discuss $\zeta^{\prime}(\alpha)$, and
compute $\zeta^{\prime}(0)$. Unfortunately, it turns out that we need more
tools to be able to calculate $\zeta^{\prime}(\alpha)$ for arbitrary $\alpha$.

\section{Preliminaries}

\subsection{The Ces\`aro concept in distributions}

\subsubsection{Ces\`{a}ro summation and Ces\`{a}ro integration}

Ces\`{a}ro summation and Ces\`{a}ro integration are classical concepts in
analysis. For a detailed treatment, the reader may refer to
\cite{Hardy}. Here we just recall the definitions and basic properties.

\begin{definition}
\cite{Estrada,Hardy} For an arbitrary series $\sum_{n=0}^{\infty}{a_{n}}$ with
partial sums $A_{n}=\sum_{j=0}^{n}{a_{j}}$, define $A_{n}^{0}=A_{n}$, and
\[
A_{n}^{k+1}=A_{0}^{k}+A_{1}^{k}+\cdots+A_{n}^{k}.
\]
If $\lim_{n\to\infty}{C_{n}^{k}}=\lim_{n\to\infty}{A_{n}^{k}/\binom{n+k}{k}%
}=S$, then the series $\sum_{n=0}^{\infty}{a_{n}}$ is said to be $k$-th order
Ces\`aro summable to $S$, denoted as
\[
\sum_{n=0}^{\infty}{a_{n}}=S\quad(C,k).
\]
Note that $\binom{n+k}{n}\sim\frac{n^{k}}{k!}$ as $n\to\infty$, then $S$ is
generally also expressed as $\lim_{n\to\infty}{C_{n}^{k}}=\lim_{n\to\infty
}{\frac{k!A_{n}^{k}}{n^{k}}}=S$.
\end{definition}

Still take the series $\sum_{n=0}^{\infty}{(-1)^{n}}$ as an example. Its
partial sums satisfy
$$A_{n}=%
\begin{cases}
1,n=2m-1\\
0,n=2m
\end{cases}
,\quad A_{n}^{1}=%
\begin{cases}
\frac{n+1}{2}, & n=2m-1\\
\frac{n}{2}, & n=2m
\end{cases}.$$
Obviously, $\lim_{n\rightarrow\infty}C_{n}^{1}=\lim_{n\rightarrow\infty}%
\frac{A_{n}^{1}}{n}=\frac{1}{2}$, that is, $\sum_{n=0}^{\infty}{(-1)^{n}%
}=\frac{1}{2}\quad(C,1)$. This is consistent with the result obtained by using
the generalized sum of power series. In addition, when a series converges, the
$k$-th Ces\`{a}ro sum of it for any $k$ is equal to the sum of the convergent
series. This kind of generalized sum is also called Ces\`{a}ro mean summation.
For example, from the second Ces\`{a}ro sum one could easily compute that 
$\sum_{n = 0}^{\infty} (-1)^{n} n = -\frac{1}{4}\quad(C,2)$.

However, the Ces\`{a}ro summation method still has strict constraints. For example, for the power series $\sum_{n=0}^{\infty}2^{n}$, the order of magnitude of $A_{n}
^{k}$ is always much larger than $n^{k}$, thus this series does not have a
generalized sum under the Ces\`{a}ro definition. In fact, one can see that a
necessary condition for Ces\`{a}ro summability is $a_{n}=O(n^{k})$.

For the Riemann zeta function $\zeta(s)$, in the right-half plane
$\mathrm{Re}(s)>1$ of the complex plane, it has the form of series
$\sum_{n=1}^{\infty}{\frac{1}{n^{s}}}$. While in the region $\mathrm{Re}(s)<1$, it is the corresponding analytic continuation. For example, when
$s=-1$, people sometimes write down the equation
\begin{equation}
1+2+3+\cdots+n+\cdots=-\frac{1}{12}\quad\label{2.1}%
\end{equation}
to express the meaning that $\zeta(-1)=-\frac{1}{12}$ \cite{TEC}.

In fact, it is well-known that the Riemann zeta function $\zeta(s)$ satisfies
the functional equation
\[
\zeta(s)=2^{s}\pi^{s-1}\sin\left(  \frac{\pi s}{2}\right)
\Gamma(1-s)\zeta(1-s).
\]
By the uniqueness of the analytic continuation, (\ref{2.1}) can be found by
any suitable generalized sum of the divergent series. Unfortunately, although
the series in (\ref{2.1}) satisfies the necessary condition for Ces\`{a}ro
summability, no matter what order $k$ is, it is difficult to obtain this
result via the Ces\`{a}ro sum. Of course, one has other summations such as the Ramanujan summation etc. While the concept of Ces\`{a}ro limit of distributions introduced by Estrada and Kanwal \cite{Estrada} could be applied here to get some short and direct proofs. Let us first recall the classical definition of the Ces\`{a}ro integral.

\begin{definition}
\cite{Estrada,LA} For an infinite integral $\int_{0}^{\infty}{f(x)\mathrm{d}%
x}$, if $\lim_{x\to\infty}\int_{0}^{x}{(1-t/x)^{k}f(t)\mathrm{d}t}=I$ exists,
then the $k$-th order Ces\`aro mean of the infinite integral $\int_{0}%
^{\infty}{f(x)\mathrm{d}}x$ is defined as $I$, where $k\in\mathbb{R} $ and
$k>-1$. This is denoted by
\[
\int_{0}^{\infty}{f\left(  x \right)  \mathrm{d}x}=I\quad(C,k).
\]

\end{definition}

Note that here $k$ is not necessarily an integer. When $k$ is an integer, let
$F_{0}(x) = \int_{0}^{x} f(t) \mathrm{d}t$, and $F_{k + 1}(x) = \int_{0}^{x}
F_{k}(t) \mathrm{d}t$, that is, $F_{k}(x)$ is the $k$-th order primitive
function of $F_{0}(x)$. Then it satisfies
\begin{equation}
\label{2.2}\lim_{x \to\infty} \frac{k! F_{k}(x)}{x^{k}} = I.
\end{equation}

Essentially, both the Ces\`{a}ro integral and Ces\`{a}ro summation are
definitions under the idea of mean value. A warm-up example would be that
$\int_{0}^{\infty}\sin(ax)\mathrm{d}x$ is 1-st order Ces\`{a}ro integrable,
since $\lim_{x\rightarrow\infty}\frac{ax-\sin(ax)}{a^{2}x}=\frac{1}{a}.$

\subsubsection{Ces\`aro limit}

\begin{definition}
\cite{Estrada}\label{def3} Let $\beta\in\mathbb{R}\setminus\{-1,-2,\cdots\}$.
We say that
\begin{equation}
\label{2.3}f(x)=O(x^{\beta})\quad(C),\quad x\to\infty,
\end{equation}
if there exists a positive integer $N$ and a polynomial $p(x)$ of degree $N-1
$ such that the $N$-th order primitive function $F(x)$ of $f(x)$ (the
primitive function generally takes $\int_{0}^{x}{f(t)\mathrm{d}t}$) is locally
integrable and satisfies
\[
F(x)=p(x)+O(x^{N+\beta}),\quad x\to\infty.
\]
Given an $N$, (\ref{2.3}) can also be expressed as:
\[
f(x) = O(x^{\beta}) \quad(C, N), \quad x \to\infty.
\]

\end{definition}

When $\beta=\{-1,-2,\ldots\}$, a similar definition to (\ref{2.3}) must also
be given. Let $\rho_{k,N}(x)$ be the $N$-th order primitive function of
$x^{-k}$, we have the following definition:

\begin{definition}
\cite{Estrada}\label{def4} Let $k \in\{1, 2,\cdots\}$. We say that
\[
f(x)=O(x^{-k})\quad(C),\quad x\to\infty,
\]
if there exists a positive integer $N$ and a polynomial $p(x)$ of degree $N-1
$ such that the $N$-th order primitive function $F(x)$ of $f(x)$ is locally
integrable and satisfies
\[
F(x)=p(x)+O(\rho_{k,N}(x)),\quad x\to\infty.
\]

\end{definition}

In Definition \ref{def3} and Definition \ref{def4}, the function $f(x)$ can be
any locally integrable function or a distribution in the space $\mathcal{D}%
^{\prime}$. If the big-$O$ symbol is replaced by a little-$o$ symbol, one has
a similar definition for $f(x)=o(x^{\beta})\quad(C)$. If $f(x)=o(x^{\beta
})\quad(C)$ for any $\beta$, then one denotes that $f(x)=o(x^{-\infty}%
)\quad(C)$.

From the definition, it is easy to prove the following conclusions:

\begin{proposition}
\cite{Estrada}\label{prop1} Let $f(x) \in\mathcal{D}^{\prime}$ and satisfy
$f(x) = O(x^{\beta})\quad(C), \quad x \to\infty$. Then we have

\begin{enumerate}
\item If $f(x) = O(x^{\beta}) \quad(C, N)$ and $\beta\notin\mathbb{Z}$, then
$f^{(k)}(x) = O(x^{\beta- k}) \quad(C, N + k)$.

\item If $f(x) = O(x^{\beta}) \quad(C, N)$ and $F$ is the $n$-th order
primitive function of $f$, then there exist real numbers $a_{j}$ such that
\[
F(x) = a_{n}x^{n} + \cdots+ a_{1}x + a_{0} + O(x^{\beta+ n}) \quad(C, M),
\quad x \to\infty,
\]
where $M = \max\{N - n, 0\}$.

\item Let $\alpha\in\mathbb{R}$, if $\alpha+\beta\neq-1,-2,\cdots$, then
\[
x^{\alpha}f(x)=o(x^{\alpha+\beta})\quad(C),\quad x\to\infty.
\]

\end{enumerate}
\end{proposition}

\begin{definition}
\cite{Estrada} Let $f(x) \in\mathcal{D}^{\prime}$ . We say that $f(x)$ has a
Ces\`aro limit $L$ as $x\to\infty$, if $f(x)=L+o(1)\quad(C),\quad x\to\infty.$
The Ces\`aro limit is denoted as:
\[
\lim_{x\to\infty}f(x)=L\quad(C).
\]

\end{definition}

It is not hard to prove the following relation between the Ces\`{a}ro limit in
the usual sense and in the distributional sense by (\ref{2.2}) \cite{Estrada}:

\begin{proposition}
\label{prop2} Let $F_{k}(x)$ be the $k$-th order primitive function of $f(x)$,
and suppose that $\lim_{x\rightarrow\infty}\frac{k!F_{k}(x)}{x^{k}}=L$. Then
\[
\lim_{x\rightarrow\infty}f(x)=L\quad(C).
\]

\end{proposition}

\bigskip

It is easy to see that $\sin x=o(x^{-\infty})\quad(C)$ and $\cos
x=o(x^{-\infty})\quad(C).$ In fact, we cite the following lemma for our proof
in the section \ref{sec3}. 

\begin{lemma}\label{lemma} 
A periodic function $f(x)$ with periodic mean value 0 is
$o(x^{-\infty})\quad(C)$.
\end{lemma}

\begin{proof}
Let $T$ be the period of $f$. We only need to show that $F_{1}(x)=\int_{0}^{x}f(t)\mathrm{d}t$ satisfies the
conclusion.
\[
F_{1}(x+T)=\int_{0}^{x+T}f(t)\mathrm{d}t=\int_{0}^{x}f(t)\mathrm{d}t+\int%
_{x}^{x+T}f(t)\mathrm{d}t=\int_{0}^{x}f(t)\mathrm{d}t=F_{1}(x).
\]
Then $F_{1}(x)-a_{0}$ has zero mean value, where $a_{0}$ is the mean value of
$F_{1}$. 
\end{proof}

\begin{definition}
\cite{Estrada}\label{def6} Let $f$ be a distribution and its support be
bounded on the left end of $\mathbb{R}$ with the left endpoint $a$. Let $\phi$
be a test function. The value of $\langle f(x),\phi(x)\rangle$ in the
Ces\`{a}ro sense is said to exist if
\[
\lim_{x\rightarrow\infty}\int_{a}^{x}f(t)\phi(t)\mathrm{d}t=L\quad(C).
\]
It is denoted as $\langle f(x),\phi(x)\rangle=L\quad(C)$.
\end{definition}

As one can see, essentially, Definition \ref{def6} is a generalization of the
Ces\`{a}ro limits in the theory of distributions.

\subsection{Regularization and the Riemann zeta function}

In this subsection, we briefly summarize the tools needed for our proof of the
Riemann zeta values. Most results in this subsection are from \cite{Estrada}.

\begin{definition}
\cite{Estrada} $\mathcal{K}_{q}(\mathbb{R})$ is a test function space on
$\mathbb{R}$. The elements $\phi(x)$ in $\mathcal{K}_{q}(\mathbb{R})$ are smooth
functions, and for any $\phi(x)$, there exists a $q \in\mathbb{R}$ such that
\[
D^{k}\phi(x) = O(|x|^{q-k}), \quad|x| \to\infty.
\]
$\mathcal{K}(\mathbb{R})$ is the inductive limit of $\mathcal{K}%
_{q}(\mathbb{R})$ as $q\nearrow\infty$.
\end{definition}

Recall that, here "the inductive limit" means that $f\in\mathcal{K}$ if and
only if there exists some $q$ such that $f\in\mathcal{K}_{q}$.

\begin{definition}
\cite{Estrada} The elements in $\mathcal{K}\{x^{\alpha_{n}}\}$ are smooth
functions defined on $(0,\infty)$ and bounded by polynomial growth. And for
any $\phi(x)\in\mathcal{K}\{x^{\alpha_{n}}\}$, $\phi(x)$ satisfies a strong
asymptotic expansion near the origin:
\[
\phi\left(  x\right)  \sim a_{1}x^{\alpha_{1}}+a_{2}x^{\alpha_{2}}%
+\cdots+a_{n}x^{\alpha_{n}}+\cdots,\quad x\rightarrow0^{+}.
\]

\end{definition}

Here, "strong" means the derivative of $\phi(x)$ of any order admits an
asymptotic expansion as the term-by-term derivative of the right-hand side.

\begin{theorem}
\cite[Thm.6.7.2]{Estrada}\label{thm1} If $f \in\mathcal{K}^{\prime}$ and
$\phi\in\mathcal{K}$, then the distributional evaluation $\langle f,
\phi\rangle$ is Ces\`aro summable.
\end{theorem}

A detailed proof of Theorem \ref{thm1} can be found in reference \cite{Estrada}.

Let $f(x)=\sum_{n=1}^{\infty}\delta(x-n)-H(x-1)$, where $\delta(x)$ is
the Dirac delta function and $H(x)=%
\begin{cases}
0, & x<0\\
1, & x>0
\end{cases}
$ is the Heaviside function. Then $f(x)\in\mathcal{K}^{\prime}$. Take a smoothing function $\phi_{0}\left(
x\right)  $ such that $\phi_{0}\left(  x\right)  =%
\begin{cases}
1,x>1\\
0,x<1/2
\end{cases}
$, we have $\phi_{\alpha}\left(  x\right)  =\phi_{0}\left(  x\right)  \cdot
x^{\alpha}\in\mathcal{K}$. From Theorem \ref{thm1}, denoting $Z(\alpha
)=\langle f(x),\phi_{\alpha}(x)\rangle $, then
\begin{equation}
Z(\alpha)=\sum_{n=1}^{\infty}n^{\alpha}-\int_{1}^{\infty}x^{\alpha}%
\mathrm{d}x\quad(C). \label{2.4}%
\end{equation}
Note that the two terms on the right-hand side of equation (\ref{2.4})
converge when $\mathrm{Re}(\alpha)<-1$, thus
\[
Z(\alpha)=\sum_{n=1}^{\infty}n^{\alpha}-\int_{1}^{\infty}x^{\alpha}\mathrm{d}x=\zeta(-\alpha)+\frac{1}{\alpha+1},\quad\mathrm{Re}(\alpha)<-1,
\]
where $\zeta(s)$ is the Riemann zeta function.

Take $g(x)=\sum_{n=1}^{\infty}\delta(x-n)-H(x)$, then $g(x)\in\mathcal{K}^{\prime}
\left\{  x^{\alpha_{n}}\right\}  $. Considering $\left\langle \sum
_{n=1}^{\infty}\delta(x-n)-H(x),x^{\alpha}\right\rangle $, we have
\[
\begin{aligned} \left\langle \sum_{n =1}^{\infty} \delta(x - n) - H(x), x^{\alpha} \right\rangle &= \left\langle \sum_{n =1}^{\infty} \delta(x - n) - H(x - 1) - [H(x) - H(x - 1)], x^{\alpha} \right\rangle \\ &= Z(\alpha) - \mathrm{F.p.} \int_{0}^{1} x^{\alpha}\mathrm{d}x \quad (C), \end{aligned}
\]
where $\mathrm{F.p.}\int_{0}^{1}x^{\alpha}\mathrm{d}x=\frac{1}{\alpha+1}$ for $\alpha\neq -1$.

By combining with (\ref{2.4}), we obtain
\begin{equation}
\label{2.5}\sum_{n=1}^{\infty} n^{\alpha} - \mathrm{F.p.} \int_{0}^{\infty}
x^{\alpha} \mathrm{d}x = \zeta(-\alpha) \quad(C).
\end{equation}

Note that the above formula holds under the condition $\mathrm{Re}(\alpha) <
-1 $. However, the zeta function has an analytic continuation in the complex
plane except at $\{1\}$, and the result of the Ces\`aro limit is also a smooth
extension on the real line. Due to the uniqueness of analytic continuation,
formula (\ref{2.5}) holds for any $\alpha\neq-1 $. According to Definition
\ref{def6}, equation (\ref{2.5}) can be rewritten as
\begin{equation}
\label{2.6}\lim_{x \to\infty} \left(  \sum_{n=1}^{[x]}
n^{\alpha} - \mathrm{F.p.} \int_{0}^{x} t^{\alpha} \mathrm{d}t \right)  =
\zeta(-\alpha) \quad(C),
\end{equation}
where $[x]$ is the floor function of $x$.

Note that $\mathrm{F.p.}$ means the Hadamard finite part. Let us briefly
review some basic concepts of the Hadamard finite part.

\subsection{Hadamard finite part}

\begin{definition}
\cite{BL,Estrada,Yang} Let $g(\varepsilon)$ be a function defined on $(0,
\eta)$ and $\lim_{\varepsilon\to0^{+}} g(\varepsilon) = \infty$. $\mathcal{F}
$ is a function family, where all elements of $\mathcal{F}$ are strictly
positive functions tending to infinity at $0$. For distinct $f_{1}, f_{2}
\in\mathcal{F}$, $\lim_{\varepsilon\to0^{+}} \frac{f_{1}(\varepsilon)}%
{f_{2}(\varepsilon)} = 0$ or $\infty$. If $g(\varepsilon) = g_{1}(\varepsilon)
+ g_{2}(\varepsilon)$, where $g_{1}$ can be expressed as a linear combination
of functions in $\mathcal{F}$ and $g_{2}$ satisfies that $\lim_{\varepsilon
\to0^{+}} g_{2}(\varepsilon)=A$ is finite, then $g_{1}$ is defined as the
infinite part of $g(\varepsilon)$, and $g_{2}$ is defined as the finite part
of $g(\varepsilon)$. Such a decomposition is unique because any finite number
of elements in $\mathcal{F}$ are linearly independent. Thus we can define the
finite part of $\lim_{\varepsilon\to0^{+}} g(\varepsilon)$ with respect to
$\mathcal{F}$ as $A$, and denote it by:
\begin{equation}
\label{2.7}\mathrm{F}.\mathrm{p}._{\mathcal{F}}\lim_{\varepsilon\rightarrow
0^{+}} g\left(  \varepsilon\right)  =A.
\end{equation}

\end{definition}

In the standard Hadamard finite part, take elements in $\mathcal{F}$ as
$\varepsilon^{-\alpha} |\ln\varepsilon|^{\beta}$, where $\alpha> 0$ and
$\beta\geq0$, or $\alpha= 0$, $\beta\geq0$. At this case, (\ref{2.7}) is
abbreviated as:
\begin{equation}
\mathrm{F}.\mathrm{p}.\lim_{\varepsilon\rightarrow0^{+}} g\left(
\varepsilon\right)  =A.
\end{equation}

It is obvious that the standard Hadamard finite part integral is a
generalization of the ordinary integral. For example:
\[
\int_{0}^{1} x^{\alpha}\,\mathrm{d}x = \mathrm{F}.\mathrm{p}.\lim
_{\varepsilon\to0^{+}} \int_{\varepsilon}^{1} x^{\alpha}\,\mathrm{d}x =
\mathrm{F}.\mathrm{p}.\lim_{\varepsilon\to0^{+}} \frac{1 - \varepsilon
^{\alpha+ 1}}{\alpha+ 1}= \frac{1}{\alpha+ 1},\ (\alpha\neq-1)
\]
and
\[
\int_{0}^{1} x^{-1} \,\mathrm{d}x =\mathrm{F}.\mathrm{p}.\lim_{\varepsilon
\to0^{+}} \int_{\varepsilon}^{1} x^{-1} \,\mathrm{d}x =\mathrm{F}%
.\mathrm{p}.\lim_{\varepsilon\to0^{+}} (\ln1 - \ln\varepsilon) = 0.
\]
It is known that when $\alpha> -1$, the value of the ordinary integral
$\int_{0}^{1} x^{\alpha}\,\mathrm{d}x$ is also $\frac{1}{\alpha+ 1}$.
Therefore, $\mathrm{F}.\mathrm{p}.\int_{0}^{1} x^{\alpha}\,\mathrm{d}x$
defines a smooth extension of the function $\int_{0}^{1} x^{\alpha
}\,\mathrm{d}x$ except at $\{-1\} $.

\section{The value of the Riemann zeta function}

\label{sec3}

\subsection{Evaluation of the zeta function}

Let us compute two toy examples.

\begin{example}
Evaluate $\zeta(0)$.\newline Substituting $\alpha=0$ in (\ref{2.6}), the
left-hand side becomes $\lim_{x\to\infty}([x]-x)$. Noting that $[x]-x$ is a
periodic function with period $1$, we apply Lemma \ref{lemma} to calculate:
\[
\lim_{x \to\infty} ([x] - x)= \int_{0}^{1}([x]-x)\mathrm{d}x = -\frac{1}{2}
\quad(C).
\]
Thus we obtain $\zeta(0)=-\frac{1}{2}$.
\end{example}

\begin{example}
Evaluate $\zeta(-1)$.\newline Similarly, substitute $\alpha=1$ into
(\ref{2.6}). Let $f(x)=\sum_{n = 1}^{[x]}n^{1}-\mathrm{F.p.}\int_{0}^{x}%
t^{1}\mathrm{d}t=\frac{[x]^{2} + [x]-x^{2}}{2}$, and let $F_{n}$ be its $n$-th order
primitive function. Then
\[
\begin{aligned} F_1(x)&=\int_{0}^{[x]}f(t)\mathrm{d}t+\int_{[x]}^{x}f(t)\mathrm{d}t\\ 
	&=\sum_{k = 0}^{[x]-1}\int_{k}^{k + 1}f(t)\mathrm{d}t+\int_{[x]}^{x}f(t)\mathrm{d}t\\
	&=\frac{1}{2}\left(([x]^2 + [x])(x - [x])-\frac{x^3- [x]^3}{3}-\frac{[x]}{3}\right). \end{aligned}
\]
Let $x = [x]+y:=n + y$ and write
\[
F_{1}(x)=\frac{1}{2}\left(  n( y- y^{2}-\frac{1}{3})
-\frac{y^{3}}{3}\right)  .
\]
The ordinary limit of $\frac{F_{1}(x)}{x}$ still does not exist. Hence we compute
\[
\begin{aligned} F_2(x) & = \sum_{k=0}^{[x]-1} \int_k^{k+1} F_1(t)\mathrm{d}t + \int_{[x]}^x F_1(t)\mathrm{d}t \\
	 &= -\frac{n^2}{24} + \frac{1}{2} \int_n^{n+y} \left( (n^2 + n)(t - n) - \frac{t^3 - n^3}{3} - \frac{n}{3} \right) \mathrm{d}t \\ &= -\frac{n^2}{24} + \frac{1}{2} \left( n( \frac{y^2}{2} - \frac{y^3}{3} - \frac{y}{3}) - \frac{y^4}{12} \right)
	. \end{aligned}
\]
Since $\lim_{x \to\infty} \frac{n}{x} = \lim_{x \to\infty} \frac{[x]}{x} = 1$,
then $\lim_{x \to\infty} \frac{2! F_{2}(x)}{x^{2}} = -\frac{1}{12}$. It
follows from Proposition \ref{prop2} that $\lim_{x \to\infty} f(x) = -\frac
{1}{12} \quad(C)$. Thus we obtain $\zeta(-1)=-\frac{1}{12}$.
\end{example}

Let us now use Lemma \ref{lemma} to derive a general formula for
$\zeta(-n)$.

\begin{theorem}
\label{Thm} We have the values of the zeta function at negative integers:
\[
\zeta(-n)=-\frac{B_{n+1}}{n+1},
\]
where $B_{n}$ are the Bernoulli numbers.
\end{theorem}

Before proving this theorem, let us recall the Bernoulli numbers. The
Bernoulli numbers $\{B_{n}\}_{n=0}^{\infty}$ are defined by the expansion of
the analytic function $\frac{z}{e^{z}-1}$:
\[
\frac{z}{e^{z}-1}=\sum_{n=0}^{\infty}{B_{n}\frac{z^{n}}{n!}}.
\]
The initial terms are $B_{0}=1, B_{1}=-\frac{1}{2}, B_{2}=\frac{1}{6},
B_{3}=0, B_{4}=-\frac{1}{30}, B_{5}=0\cdots$ It is easy to prove that the
Bernoulli numbers $\{B_{n}\}_{n=0}^{\infty}$ satisfies the recurrence
relation: $B_{0}=1$ and
\begin{equation}
\label{re1}\sum_{k=0}^{n}{\binom{n+1}{k} B_{k}}=0, n\geqslant1
\end{equation}
or
\begin{equation}
\label{re2}B_{n}=-\frac{1}{n+1}\sum_{k=0}^{n-1}{\binom{n+1}{k}B_{k}},
n\geqslant1.
\end{equation}
The famous Faulhaber--Bernoulli formula gives the expression for the sum of
powers of natural numbers using Bernoulli numbers:
\begin{equation}
\label{F-B}\sum_{k=1}^{m-1}{k^{n}}=\frac{1}{n+1}\sum_{k=0}^{n}{\binom{n+1}%
{k}B_{k}m^{n-k+1}}, n\geqslant1.
\end{equation}
We now present a new proof of Theorem \ref{Thm}:

\begin{proof}
Denote
\[
f(x)=\sum_{k=1}^{\left[  x \right]  }{k^{n}}-\mathrm{F.p.}\int_{0}^{x}%
{t^{n}\mathrm{d}t}.
\]
From (\ref{2.6}), we need to calculate $\lim_{x \to\infty}f(x)\quad(C)$.

Using Faulhaber--Bernoulli formula (\ref{F-B}), we have
\[
\begin{aligned} f(x) &=\frac{1}{n+1}\sum_{k=0}^n{\binom{n+1}{k} B_k\left( \left[ x \right] +1 \right) ^{n-k+1}}-\frac{x^{n+1}}{n+1}\\ &=\frac{1}{n+1}\sum_{k=0}^n{\binom{n+1}{k} B_k\left( x-\left\{ x \right\}+1 \right) ^{n-k+1}}-\frac{x^{n+1}}{n+1}\\ &:=\sum_{m=0}^n{P_m\left( \left\{ x \right\} \right) x^m} \end{aligned}
\]
where $\{x\}=x-[x]$ and $P_{m}$ is a polynomial. Hence all $P_{m}(\{x\})$ are
functions with period $1$. In fact, we have
\[
P_{0}\left(  \left\{  x \right\}  \right)  =\frac{1}{n+1}\sum_{k=0}^{n}%
{\binom{n+1}{k} B_{k}\left(  1-\left\{  x \right\}  \right)  ^{n-k+1}}%
\]
and
\[
P_{m}\left(  \left\{  x \right\}  \right)  =\sum_{k=0}^{n-m+1}{\binom{n+1}{k}
\binom{n-k+1}{m} B_{k}\left(  1-\left\{  x \right\}  \right)  ^{n-k-m+1}},
1\leqslant m\leqslant n.
\]
When $1\leqslant m\leqslant n$, the periodic mean value of $P_{m}{(x)}$ is
\[
\begin{aligned} \int_0^1{P_m\left( \left\{ x \right\} \right) \mathrm{d}x} &=\sum_{k=0}^{n-m+1}{\binom{n+1}{k}\binom{n-k+1}{m} \frac{B_k}{n-k-m+2}}\\ &=\frac{1}{n+2}\binom{n+2}{m} \sum_{k=0}^{n-m+1}{\binom{n-m+2}{k} B_k}\\ &\xlongequal{(\ref{re1})}0. \end{aligned}
\]
It follows from Lemma \ref{lemma} that $P_{m}({x})=o(x^{-\infty})$ for
$1\leqslant m\leqslant n$. According to the properties of the Ces\`aro limit,
the Ces\`aro limit of $f(x)$ as $x \to\infty$ equals
\[
\begin{aligned} \int_0^1{P_0\left( \left\{ x \right\} \right) \mathrm{d}x} &=\frac{1}{n+1}\sum_{k=0}^n{\binom{n+1}{k} \frac{B_k}{n-k+2}}\\ & =\frac{1}{\left( n+1 \right) \left( n+2 \right)}\sum_{k=0}^n{\binom{n+2}{k} B_k}\\ &\xlongequal{(\ref{re2})}-\frac{B_{n+1}}{n+1}. \end{aligned}
\]
Thus we obtain $\zeta(-n)=-\frac{B_{n+1}}{n+1}$, where $n$ is a positive integer.
\end{proof}

\subsection{Some discussions about $\zeta^{\prime}(\alpha)$}

The same method can be used to construct the analytic continuation of the
derivative of the Riemann zeta function with only a modification of
$\phi_{\alpha}$. Still, let $f(x) = \sum_{n=1}^{\infty} \delta(x - n) - H(x -
1) $ and define $\phi_{\alpha}= \phi_{0} \cdot\ln x \cdot x^{\alpha}$. For any
$\mathrm{Re}(\alpha) < -1$, we have
\[
\left\langle \sum_{n=1}^{\infty} \delta(x - n) - H(x - 1), \phi_{\alpha}
\right\rangle = Z(\alpha) \quad(C),
\]
where $Z(\alpha) = -\zeta^{\prime}(-\alpha) - \frac{1}{(1 + \alpha)^{2}}$.\\
Taking another function $f(x)=\sum_{n=1}^{\infty} \delta(x - n) - H(x)$, we
also have
\[
\left\langle \sum_{n=1}^{\infty} \delta(x - n) - H(x), \ln x \cdot x^{\alpha}
\right\rangle = Z(\alpha) - \mathrm{F.p.} \int_{0}^{1} \ln x \cdot x^{\alpha}
\mathrm{d} x \quad(C),
\]
and noting that
\[
\mathrm{F.p.} \int_{0}^{1} \ln x \cdot x^{\alpha} \mathrm{d}x = \mathrm{F.p.}
\lim_{A \to0} \int_{A}^{1} \ln x \cdot x^{\alpha} \mathrm{d}x = -\frac
{1}{(\alpha+ 1)^{2}},
\]
we conclude:
\begin{equation}
\label{3.4}-\zeta^{\prime}(-\alpha)  = \lim_{x \to\infty} \left(
\sum_{n=1}^{[x]} \ln n \cdot n^{\alpha} - \mathrm{F.p.} \int_{0}^{x} \ln t
\cdot t^{\alpha} \mathrm{d}t \right)  \quad(C).
\end{equation}

We can compute the value of $\zeta^{\prime}(0)$ from (\ref{3.4}):

\begin{example}
Evaluate $\zeta^{\prime}(0)$.\newline Take $\alpha=0$ in (\ref{3.4}), we
obtain $\zeta^{\prime}(0)=-\lim_{x\rightarrow\infty}\left(  \sum
_{n=1}^{[x]}\ln n-\mathrm{F.p.}\int_{0}^{x}\ln t\mathrm{d}t\right)  \quad(C)$. For
simplicity, define
\[
f(x)=\mathrm{F.p.}\int_{0}^{x}\ln t\mathrm{d}t-\sum_{n=1}^{[x]}\ln n=x\ln
x-x-\ln([x]!)=\ln\left(  \frac{x^{x}}{e^{x}\cdot\lbrack x]!}\right)  .
\]
By setting $x=[x]+y:=n+y$, and applying Stirling's formula, we have
\[
f(x)=\ln\left(  \frac{x^{x}}{e^{x}\cdot n!}\right)  \sim\ln\frac{x^{x}}%
{e^{x}\cdot\sqrt{2\pi n}\left(  \frac{n}{e}\right)  ^{n}}=\ln\left(
\frac{\left(  1+\frac{y}{n}\right)  ^{n}x^{y}}{e^{y}\cdot\sqrt{2\pi n}%
}\right)  \sim\ln\left(  \frac{x^{y}}{\sqrt{2\pi n}}\right)
\]
as $x\rightarrow\infty$. Define $g(x)=\ln\left(  \frac{x^{y}}{\sqrt{2\pi n}%
}\right)  $, then $f(x)\sim g(x)$ as $x\rightarrow\infty$ and
\[
g(x)=-\frac{1}{2}\ln(2\pi)+\left(  x-[x]-\frac{1}{2}\right)  \ln x+\frac{1}%
{2}\ln\frac{x}{[x]}.
\]
Note that $x-[x]-\frac{1}{2}$ is a periodic function with zero mean value,
and $\ln x=o(x)$ as $x\rightarrow\infty$. It follows from Lemma \ref{lemma} that
\[
\left(  x-[x]-\frac{1}{2}\right)  \ln x=o(x^{-\infty})\quad(C),x\rightarrow
\infty.
\]
Furthermore, since $\lim_{x\rightarrow\infty}\ln\frac{x}{[x]}=0$, we have
$g(x)=-\frac{1}{2}\ln(2\pi)+o(1)\quad(C),x\rightarrow\infty$. Recalling that $f(x) \sim g(x)$ as  $x\rightarrow\infty$, we obtain $\lim_{x \to \infty} f(x) = -\frac{1}{2} \ln(2\pi)\quad (C)$. Substituting into the expression for $\zeta'(0)$ yields $\zeta'(0) = -\frac{1}{2} \ln(2\pi)$.
\end{example}

\section*{Acknowledgements}

Thanks to the National Natural Science Foundation of China to support this
research. The grant number is 12001150.

\bibliographystyle{plain}
\bibliography{reference}

\end{document}